\documentclass[12pt]{article}
\usepackage{amsmath,amsfonts,amssymb,amsthm,color}
\usepackage{epsfig}
\usepackage{fancyhdr}
\usepackage{framed}
\usepackage{url}

\newcounter{abci}\renewcommand{\theabci}{\alph{abci}}

\theoremstyle{plain}
\newtheorem{lem}{Lemma}
\newtheorem{thm}{Theorem}
\newtheorem{prop}{Proposition}
\newtheorem{Def}{Definition}
\newtheorem{cor}{Corollary}

\theoremstyle{definition}
\newtheorem{ex}{Example}

\newtheorem{Rem}{Remark}

\title{A new life of Pearson's skewness}

\author{Yevgeniy Kovchegov
\footnote{Department of Mathematics, Oregon State University, Corvallis, OR 97331, USA 
\newline
Email: {\it kovchegy@math.oregonstate.edu}}}

\date{}

\begin{document}

\maketitle

\begin{abstract}
In this work we show how coupling and stochastic dominance methods can be successfully applied to 
a classical problem of rigorizing Pearson's skewness. Here, we use Fr\'{e}chet means to define generalized notions of positive and negative skewness
that we call {\it truly positive} and {\it truly negative}.
Then, we apply stochastic dominance approach in establishing criteria for determining whether a continuous random variable is truly positively skewed.
Intuitively, this means that scaled right tail of the probability density function exhibits strict stochastic dominance over equivalently scaled left tail.
Finally, we use the stochastic dominance criteria and establish some basic examples of true positive skewness, 
thus demonstrating how the approach works in general.
\end{abstract}

\section{Introduction}\label{sec:Intro}
Consider a positively skewed (right-skewed) unimodal distribution with finite second moment. It is expected that the mode, the median, and the mean line up 
in an increasing order, i.e.,
$$\mathrm{mode}~<~\mathrm{median}~<~\mathrm{mean}.$$
Symmetrically, for a negatively skewed (left-skewed) unimodal distribution with finite second moment,
$$\mathrm{mean}~<~\mathrm{median}~<~\mathrm{mode}.$$
Here, when saying that a distribution is positively skewed we usually mean {\it Pearson's moment coefficient of skewness} (the standardized third central moment) is positive.
However, there are notable exceptions to the above orderings, also known as the {\it mean-median-mode inequalities}. See \cite{Abadir2005,Stoyanov2014}.
It all depends on how we measure skewness. 
Indeed, there are other measurements of skewness besides the moment coefficient such as Pearson's {\it first skewness coefficient} also known as {\it mode skewness} defined as
$${\mathrm{mean} -\mathrm{mode} \over \mathrm{standard~deviation}}$$
and Pearson's {\it second skewness coefficient} also known as {\it median skewness} defined as
$$3 \times {\mathrm{mean} -\mathrm{median} \over \mathrm{standard~deviation}}.$$
The sign of the latter two measurements of skewness is consistent with the above discussed ordering of the mode, the median, and the mean.

Thus, whether the signs of the mode skewness, the median skewness, and the moment coefficient of skewness align 
depends on how we define skewness. 
First, we need an approach that unifies the different measurements of skewness in determining the sign of skewness
(this is done in Def.~\ref{def:TPS} and \ref{def:TPSmode} below).
Second, when defining positive skewness, the objective is to characterize the distributions 
for which the left tail is ``spreading short" and the right tail is ``spreading longer". 
If looking at this problem from the stochastic dominance perspective, we want to dissect the distribution into the left and the right parts so that when 
we reflect the left part, align both at zero, and multiply each by a certain monotone function (positive constant times a power of $x$) so that each part 
is converted into a distribution over positive half-line.
If the distribution obtained from the right part exhibits stochastic dominance over the distribution obtained from the left part, then this should imply that 
the left tail is ``spreading short" and the right tail is ``spreading longer".
Hence, such distribution is positively skewed (reverse left and right for the negatively skewed distributions).
Thus it is natural to consider finding criteria for the skewness to be positive or negative that is based on stochastic dominance of one tail over the other.
Definition~\ref{def:TPS} of {\it true positive skewness} and its variation Def.~\ref{def:TPSmode} for unimodal distributions 
yield such stochastic dominance criteria which we establish in Sect.~\ref{sec:StocDom}, and apply in Sect.~\ref{sec:ExPosSkew}.

\medskip
\noindent
We will generalize positive and negative skewness by consider the following class of centroids known as Fr\'{e}chet $p$-means \cite{Barbaresco}.
\begin{Def}\label{def:means}
For $p\in [1,\infty)$ and a random variable $X$ with the finite $p$-th moment, the quantity
\begin{equation}\label{eqn:p-mean}
\nu_p={\rm argmin}_{a \in \mathbb{R}} E\big[|X-a|^p\big]
\end{equation}
is called {\bf Fr\'{e}chet $p$-mean}, or simply the {\bf $p$-mean}. 
\end{Def}

\noindent
The theoretical $p$-mean $\nu_p$ is uniquely defined for all $p\geq 1$ as $E\big[|X-a|^p\big]$ is a strictly convex function of $a$. 
Moreover, the $p$-mean $\nu_p$ is a unique solution of
\begin{equation}\label{eqn:nupDefAlt}
E\big[(X-\nu_p)_+^{p-1}\big]=E\big[(\nu_p-X)_+^{p-1}\big].
\end{equation}
Notice that \eqref{eqn:nupDefAlt} can be used as an extended definition of the $p$-mean that only requires finiteness of the $(p-1)$-st moment.
\begin{Def}\label{def:meansAlt}
For $p\in [1,\infty)$ and a continuous random variable $X$ with the finite $(p-1)$-st moment, the unique solution $\nu_p$ of 
\eqref{eqn:nupDefAlt} is the {\bf $p$-mean} of $X$. 
\end{Def}
\noindent
For the rest of the paper the {\bf $p$-mean} $\nu_p$ of $X$ is as defined in Def.~\ref{eqn:nupDefAlt}, and we let 
$$\mathcal{D}=\big\{p \geq 1 \,:\, E[|X|^{p-1}]<\infty \big\}$$ 
denote the domain of $\nu_p$. 
Next, we observe that $\nu_1$ and $\nu_2$ are respectively the median (in continuous case) and the mean of $X$. 
If the distribution of $X$ is unimodal, we let $\nu_0$ denote the mode. 
In the unimodal case, we will use the domain $\,\mathcal{D}_0=\mathcal{D} \cup \{0\}$.

\medskip
\noindent
For $p \in (0,1)$, there are examples of no uniqueness of $\nu_p$ in \eqref{eqn:p-mean}. For instance, if $X$ is a Bernoulli random variable with parameter $1/2$, there will be two values of $\nu_p$ for each $p \in (0,1)$. 
In Sect.~\ref{sec:p01}, we will observe that often Definition~\ref{def:meansAlt} can be extended to include $p\in(0,1)$.

\medskip
\noindent
Next, we observe that the $p$-means $\nu_0$, $\nu_1$, $\nu_2$, and $\nu_4$ are definitive for the notion of {\it positive skewness} 
defined via the Pearson's first and second skewness coefficients, as well as the Pearson's moment coefficient of skewness.
Indeed, Karl Pearson's first skewness coefficient (mode skewness) for a unimodal random variable is expressed as
${\nu_2 -\nu_0 \over \sigma}$,
where $\sigma$ denotes the standard deviation,
while Pearson's second skewness coefficient (median skewness) is given by 
${3(\nu_2 -\nu_1) \over \sigma}$.
Finally, the renown Pearson's moment coefficient of skewness 
$$\gamma=E\left[\left({X-\nu_2 \over \sigma}\right)^3\right]$$
is also related to the $p$-means via the following proposition.
\begin{prop}\label{prop:first}
Consider a random variable $X$ with a finite third moment.
The Pearson's moment coefficient of skewness $\gamma$ is positive if and only if $\nu_4>\nu_2$.
\end{prop}
\begin{proof}
Notice that the following cubic equation in $a$, 
\begin{equation}\label{eqn:cubic1}
E\big[(X-a)^3\big]=0,
\end{equation}
can be rewritten as
\begin{equation}\label{eqn:cubic2}
(a-\nu_2)^3+3\sigma^2(a-\nu_2)-\sigma^3 \gamma=0.
\end{equation}
The 4-mean $\nu_4$ is the only real root of \eqref{eqn:cubic1}, and thus of \eqref{eqn:cubic2}. Therefore, $\nu_4$  is now being obtained from \eqref{eqn:cubic2}
as follows:
\begin{equation}\label{eqn:cubic3}
\gamma=\left({\nu_4 -\nu_2 \over \sigma}\right)^3+3\left({\nu_4 -\nu_2 \over \sigma}\right).
\end{equation}
\end{proof}


\bigskip
\noindent
Throughout the rest of the paper, we consider a continuous random variable $X$ with density function $f(x)$.
Furthermore, we suppose $f(x)$ has support $supp(f)=(L,R)$, where $L$ may take value at  $-\infty$ and $R$ may take value at $\infty$.

\medskip
\noindent
By Proposition \ref{prop:first},  the second skewness coefficient and the moment coefficient of skewness are both positive if and only if
$$\nu_1 ~<~\nu_2 ~<~\nu_4.$$
In the case of unimodal continuous distribution, both, the moment coefficient of skewness is positive and the mean-median-mode inequality holds if and only if
$$\nu_0 ~<~\nu_1 ~<~\nu_2 ~<~\nu_4.$$ 
See Fig.~\ref{fig:LogNorm}. This observation suggests the notion of {\it true positive skewness} defined below.
\begin{Def}\label{def:TPS}
We say that a random variable $X$, or its distribution, is {\it truly positively skewed} if and only if $\nu_p$ is an increasing function of $p$ in the domain $\mathcal{D}$, provided the interior of $\mathcal{D}$ is nonempty.
Analogously, $X$ is {\it truly negatively skewed} if and only if $\nu_p$ is a decreasing function of $p$ in $\mathcal{D}$.
\end{Def}
\noindent
The above defined {\it true positive skewness} insures $\,\nu_1 \,<\,\nu_2 \,<\,\nu_4$.
\begin{Def}\label{def:TPSmode}
We say that a unimodal distribution is {\it truly mode positively skewed} if and only if $\nu_p$ is an increasing function of $p$ in the domain $\mathcal{D}_0$.
Analogously, it is {\it truly mode negatively skewed} if and only if $\nu_p$ is a decreasing function of $p$ in $\mathcal{D}_0$.
\end{Def}
\noindent
Observe that {\it true mode positive skewness} guarantees $\,\nu_0 \,<\,\nu_1 \,<\,\nu_2 \,<\,\nu_4$.
Notice also that a distribution can be {\it truly positively skewed} or {\it truly mode positively skewed} even in the absence of finite second moment.

\bigskip
Besides obtaining the sign of skewness, Definitions \ref{def:TPS} and \ref{def:TPSmode} do not immediately provide 
a way of measuring the magnitude of skewness.
Defining the corresponding measures of skewness and extending the approach to multidimensional distributions is touched upon in the discussion section (Sect.~\ref{sec:dis}). 
Importantly, this theoretical work concentrates on the problem of rigorously defining the sign of skewness. 
We do not intend to venture into statistical analysis or statistical applications of here defined concepts.  
In general, the question of centroids and their role in rigorously defining skewness that we considered in this paper has a long and interesting history; see for example \cite{Gauss1823,Pearson1895,Zwet64,Zwet79,McG86,DJd83,DJd88,AG95,AT98}. Yet, this problem can still generate new challenges for theoretical probabilists and statisticians alike.

In this paper we will show how true positive skewness (and analogously, true negative skewness) can be validated 
using the methods of  {\it coupling} and {\it stochastic dominance}. 
These methods are widely used in statistical mechanics and interacting particle systems, the theory of mixing times, and beyond.
See \cite{Hollander,KovOttoBook,Liggett,Lindvall,MullerStoyan2002} and references therein.
Specifically we will need the following well known result.
\begin{lem}[\cite{Lindvall,MullerStoyan2002}]\label{lem:StD}
Suppose $X$ and $Y$ are real  valued  random  variables  with  cumulative distribution functions denoted by $F_X$ and $F_Y$ respectively  and satisfying 
$F_X(x) \geq F_Y(x)$, i.e., $Y$ exhibits stochastic dominance over $X$, then, for any increasing function $h: \mathbb{R} \rightarrow \mathbb{R}$ we have
$E[h(Y)]\geq E[h(X)]$.
Moreover, if $F_X \not\equiv F_Y$, i.e., $Y$ exhibits strict stochastic dominance over $X$, and if $h(x)$ is strictly increasing, then $E[h(Y)]> E[h(X)]$.
\end{lem}
\noindent
Lemma \ref{lem:StD} is usually proved via a coupling argument. If $X$ and $Y$ are continuous random variables, the inequality in Lemma \ref{lem:StD} follows from the integration by parts. 
In some instances, instead of stating that one random variable exhibits (strict) stochastic dominance over another, it is more convenient to say that
one distribution or p.d.f. exhibits (strict) stochastic dominance over the other.

\begin{figure}[t] 
\centering\includegraphics[width=0.98\textwidth]{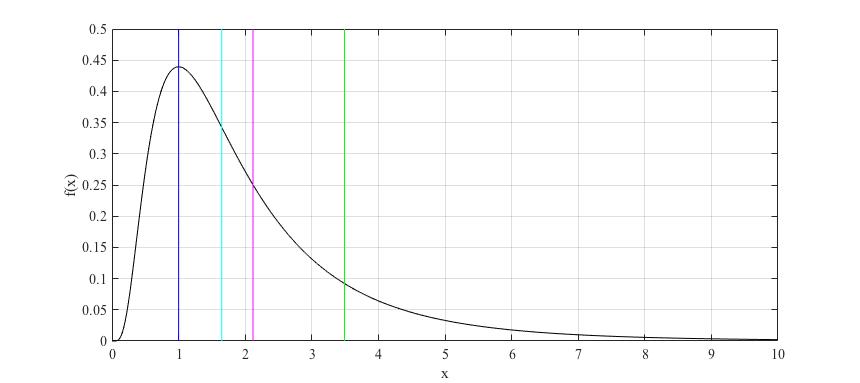}
\caption[Log-normal density function]
{For the log-normal density function with parameters $\mu=\sigma^2=1/2$ and p.d.f. $\,f(x)={1 \over \sqrt{\pi}x}\exp\left\{-(\log{x}-1/2)^2\right\}$,
we have the following centroids indicated on the plot: the mode $\,\nu_0=\exp\left\{\mu -\sigma^2\right\}=1$, the median $\,\nu_1=\exp\left\{\mu\right\}=e^{1/2}$, 
the mean $\,\nu_2=\exp\left\{\mu+{\sigma^2 \over 2}\right\}=e^{3/4}$, and the $4$-mean $\,\nu_4=\exp\left\{\mu+{3\sigma^2 \over 2}\right\}=e^{5/4}$.}
\label{fig:LogNorm}
\end{figure} 

\medskip
\noindent
Before presenting the general approach in Sect.~\ref{sec:StocDom}, we show how stochastic dominance method can be used to
establish true positive skewness of exponential random variables.
\begin{ex}[Exponential distribution]
Consider an exponential random variable $X$ with parameter $\lambda>0$.
Without loss of generality let $\lambda=1$. 
Equation \eqref{eqn:nupDefAlt} implies
$$\int\limits_0^{\nu_p} (\nu_p-x)^{p-1} e^{-x}\, dx \,= \int\limits_{\nu_p}^\infty (x-\nu_p)^{p-1} e^{-x}\, dx.$$
This simplifies to
\begin{equation}\label{eqn:expmain}
\int\limits_0^{\nu_p} x^{p-1} e^x\, dx=\Gamma(p).
\end{equation}

\medskip
\noindent
Differentiating ${d \over dp}$ in \eqref{eqn:expmain} yields
\begin{equation}\label{eqn:expmain1}
\nu_p^{p-1} e^{\nu_p}\,{d \nu_p \over dp} \,+\int\limits_0^{\nu_p} x^{p-1} e^x\,  \log{x} \, dx \,=\int\limits_0^\infty  x^{p-1} \, e^{-x}\, \log{x} \, dx
\end{equation}
Next, we observe that the gamma distribution with the p.d.f. ${1 \over \Gamma(p)}x^{p-1} \, e^{-x} {\bf 1}_{(0,\infty)}(x)$ stochastically dominates 
the distribution with the p.d.f. ${1 \over \Gamma(p)}x^{p-1} \, e^x {\bf 1}_{(0, \nu_p)}(x)$.
Thus, Lemma~\ref{lem:StD} implies
\begin{equation}\label{eqn:expmain2}
\int\limits_0^{\nu_p} x^{p-1} e^x\,  \log{x} \, dx \,< \int\limits_0^\infty  x^{p-1} \, e^{-x}\, \log{x} \, dx
\end{equation}
as $\log{x}$ is an increasing function.
Consequently, equations \eqref{eqn:expmain1} and \eqref{eqn:expmain2} imply ${d \nu_p \over dp}>0$ for all $p>0$.
Hence, exponential random variables are proved to be {\it truly positively skewed} (Def.~\ref{def:TPS}) 
and, as the mode $\nu_0=0$, {\it truly mode positively skewed} (Def.~\ref{def:TPSmode}).
\end{ex}

\section{True positive skewness via stochastic dominance}\label{sec:StocDom}
For a given $p \in \mathcal{D}$, the theoretical $p$-mean $\nu_p$ defined in \eqref{eqn:nupDefAlt} solves
\begin{equation}\label{eqn:genNu_p_Int}
H_p:=\int\limits_0^{\nu_p -L} x^{p-1} \,f(\nu_p-x)\,dx=\int\limits_0^{R-\nu_p} x^{p-1} \,f(\nu_p+x)\,dx.
\end{equation}

\medskip
\noindent
Next, we state and prove a criterion for $\nu_p$ to be increasing, and therefore for the p.d.f. $f(x)$ to be truly  positively skewed.
\begin{thm}\label{thm:TPSkew}
Consider a continuous random variable with p.d.f. $f(x)$  supported on $supp(f)=(L,R)$, and a real number $p$ in the interior of $\mathcal{D}$.
If p.d.f. ${1 \over H_p}x^{p-1}\,f(\nu_p+x){\bf 1}_{(0,R-\nu_p)}(x)$ exhibits strict stochastic dominance over 
p.d.f. ${1 \over H_p}x^{p-1}\,f(\nu_p-x){\bf 1}_{(0,\nu_p -L)}(x)$, then function $\nu_p$ is increasing at $p$.
Consequently, if the above stochastic dominance holds for all $p$ in the interior of $\mathcal{D}$, the distribution is {\it truly positively skewed} (Def.~\ref{def:TPS}).
\end{thm}
\begin{proof}
For $p$ in the interior of $\mathcal{D}$, we use Leibniz integral rule to obtain
\begin{align*}
0 =& {d \over dp}\left(\int\limits_L^{\nu_p} (\nu_p-x)^{p-1} \,f(x)\,dx ~-\int\limits_{\nu_p}^R (x-\nu_p)^{p-1} \,f(x)\,dx \right)\\
=& \left((p-1)\int\limits_L^{\nu_p} (\nu_p-x)^{p-2} \,f(x)\,dx ~+\int\limits_L^{\nu_p} (\nu_p-x)^{p-1} \,f(x)\,dx\right){d \nu_p \over dp}\\
&\quad +\left((p-1)\int\limits_{\nu_p}^R (x-\nu_p)^{p-2} \,f(x)\,dx ~-\int\limits_{\nu_p}^R (x-\nu_p)^{p-1} \,f(x)\,dx\right){d \nu_p \over dp}\\
&\quad +\int\limits_0^{\nu_p -L} x^{p-1}\, \log{x}\, f(\nu_p-x)\, dx ~-~ \int\limits_0^{R-\nu_p} x^{p-1}\,\log{x}\, f(\nu_p+x)\, dx,
\end{align*}
which simplifies to
\begin{equation}\label{eqn:dnupSimple}
{d \nu_p \over dp}={\int\limits_0^{R-\nu_p} x^{p-1}\, \log{x}\, f(\nu_p+x)\, dx -  \int\limits_0^{\nu_p -L} x^{p-1}\,\log{x}\, f(\nu_p-x)\, dx \over (p-1) \left[\int\limits_0^{\nu_p -L} x^{p-2} f(\nu_p-x)dx+ \int\limits_0^{R-\nu_p} x^{p-2}  f(\nu_p+x)\big]dx\right]}.
\end{equation}
Under the strict stochastic dominance assumption of the theorem, Lemma \ref{lem:StD} implies
\begin{equation}\label{eqn:main}
\int\limits_0^{R-\nu_p} x^{p-1}\, \log{x}\, f(\nu_p+x)\, dx ~>~ \int\limits_0^{\nu_p -L} x^{p-1}\,\log{x}\, f(\nu_p-x)\, dx.
\end{equation}
\noindent
Together, equations \eqref{eqn:dnupSimple} and \eqref{eqn:main} imply ${d \nu_p \over dp}>0$ for all $p$ in the interior of $\mathcal{D}$.
\end{proof}

\bigskip
\noindent
\begin{Rem}\label{rem:SD}
Returning to the discussion in the introduction of this paper (Sec.~\ref{sec:Intro}), Thm.~\ref{thm:TPSkew} states that if we 
dissect $f(x)$ at $\nu_p$ into the left and the right parts so that when we reflect the left part, and align both at zero,
and multiply each by ${1 \over H_p}x^{p-1}$, then each part is converted into a distribution over positive half-line,
i.e., densities ${1 \over H_p}x^{p-1}\,f(\nu_p-x){\bf 1}_{(0,\nu_p -L)}(x)$ and ${1 \over H_p}x^{p-1}\,f(\nu_p+x){\bf 1}_{(0,R-\nu_p)}(x)$.
If p.d.f. ${1 \over H_p}x^{p-1}\,f(\nu_p+x){\bf 1}_{(0,R-\nu_p)}(x)$ 
obtained from the right part exhibits stochastic dominance over p.d.f. ${1 \over H_p}x^{p-1}\,f(\nu_p-x){\bf 1}_{(0,\nu_p -L)}(x)$ obtained from the left part,
then this should imply that the left tail is ``spreading short" and the right tail is ``spreading longer".
Hence, such distribution is truly positively skewed.

Importantly, this stochastic dominance argument remains valid and true positive skewness can be established even in the case of a distribution with infinite first moment and
$\,\mathcal{D} \subseteq [1,2)$.  
See Pareto distribution example in Sect.~\ref{sec:Pareto}.
This is one of the advantages of using Def.~\ref{def:TPS} and \ref{def:TPSmode} in determining the sign of skewness.
\end{Rem}

\bigskip
\noindent
Next, we present a criterion for the p.d.f. ${1 \over H_p}x^{p-1}\,f(\nu_p+x){\bf 1}_{(0,R-\nu_p)}(x)$ to exhibit strict stochastic dominance over 
the p.d.f. ${1 \over H_p}x^{p-1}\,f(\nu_p-x){\bf 1}_{(0,\nu_p -L)}(x)$.
\begin{lem}\label{lem:StochDomCrit}
Consider a continuous random variable supported over $(L,R)$ with p.d.f. $f(x)$, and a real number $p \in \mathcal{D}$.
Suppose there exists $c>0$ such that $f(\nu_p-c)=f(\nu_p+c)$, and $f(\nu_p-x)>f(\nu_p+x)$ for $x \in (0,c)$, while $f(\nu_p-x)<f(\nu_p+x)$ for $x>c$.
Suppose also that $\nu_p-L \leq R-\nu_p$.
Then, a random variable with p.d.f. ${1 \over H_p}x^{p-1}\,f(\nu_p+x){\bf 1}_{(0,R-\nu_p)}(x)$ exhibits strict stochastic dominance over 
a random variable with p.d.f. ${1 \over H_p}x^{p-1}\,f(\nu_p-x){\bf 1}_{(0,\nu_p -L)}(x)$.
\end{lem}
\begin{proof}
Observe that $c$ is the maximum and the only local extrema of 
$$\int\limits_0^x {1 \over H_p}y^{p-1}\,f(\nu_p-y) \,dy - \int\limits_0^x {1 \over H_p}y^{p-1}\,f(\nu_p+y) \,dy$$
for $x \in [0,\nu_p-L]$, while 
$$\int\limits_0^{\nu_p-L} {1 \over H_p}y^{p-1}\,f(\nu_p-y) \,dy - \int\limits_0^{\nu_p-L} {1 \over H_p} y^{p-1}\,f(\nu_p+y) \,dy \geq 0,$$
where $\int\limits_0^{\nu_p-L} {1 \over H_p}y^{p-1}\,f(\nu_p-y) \,dy=1$.
Hence,
$$\int\limits_0^x {1 \over H_p} y^{p-1}\,f(\nu_p-y) \,dy - \int\limits_0^x {1 \over H_p} y^{p-1}\,f(\nu_p+y) \,dy >0$$
for all $x \in(0, R-\nu_p)$.
\end{proof}

\bigskip
\noindent
The following simple criterion follows immediately from the definition of stochastic dominance.
\begin{prop}\label{prop:StochDomCritEZ}
Consider a continuous random variable supported over $(L,R)$ with p.d.f. $f(x)$.
Suppose $f(x)$ is a decreasing function for $x \in (L,R)$.
Then, for all $p \in \mathcal{D}$, density function ${1 \over H_p}x^{p-1}\,f(\nu_p+x){\bf 1}_{(0,R-\nu_p)}(x)$ exhibits strict stochastic dominance over 
density function ${1 \over H_p}x^{p-1}\,f(\nu_p-x){\bf 1}_{(0,\nu_p -L)}(x)$.
\end{prop}

\section{Examples of true positive skewness}\label{sec:ExPosSkew}
In this section, we will use stochastic dominance method and the toolbox developed in Sect.~\ref{sec:StocDom}
for establishing true positive skewness for the gamma distribution, the unimodal beta distribution with the mode in the first half interval,
the log-normal distribution, and Pareto distribution.
\subsection{Gamma distribution}\label{sec:gamma}
Consider a gamma random variable with p.d.f. $$f(x)={1 \over \Gamma(\alpha)}\lambda^\alpha x^{\alpha-1} e^{-\lambda x}$$ 
over $(L,R)=(0,\infty)$, with parameters $\,\alpha >0\,$ and $\,\lambda>0$.  
We will consider cases $\,\alpha >1\,$ and $\,0<\alpha <1\,$ separately.

\bigskip
\noindent
{\bf Case i:} $\alpha >1$. Here, the mode equals $\,\nu_0={\alpha -1 \over \lambda}$.
Differentiating 
$$\log f(\nu_p-x)-\log f(\nu_p+x)=2\lambda x +(\alpha-1)\log(\nu_p-x)-(\alpha-1)\log(\nu_p+x)$$
with respect to $x$ and setting it equal to zero yields
\begin{equation}\label{eqn:cMinGamma}
x^2=\nu_p^2-{(\alpha-1)\nu_p \over \lambda}=\nu_p(\nu_p -\nu_0).
\end{equation}
\noindent
Suppose $0<\nu_p \leq \nu_0$ for some $p$, then \eqref{eqn:cMinGamma} has no positive real solution. Thus, since $f(0)=0<f(2\nu_p)$, we have  $f(\nu_p-x)< f(\nu_p+x)$ for all $x>0$. Therefore, assumption $0<\nu_p \leq \nu_0$ contradicts \eqref{eqn:genNu_p_Int}.

\medskip
\noindent
Therefore, $\nu_p> \nu_0$ for all $p \in \mathcal{D}$. Function ${f(\nu_p-x)\over f(\nu_p+x)}$ equals $1$ at $x=0$ and equals $0$ at $x=\nu_p$ with the only extremum at
$c^*=\sqrt{\nu_p(\nu_p -\nu_0)}$. Since by \eqref{eqn:genNu_p_Int}, function ${f(\nu_p-x)\over f(\nu_p+x)}$ cannot be $\leq 1$ for all $x>0$, the extremum $c^*$ 
is the location of the maximum of ${f(\nu_p-x)\over f(\nu_p+x)}$.
 Function ${f(\nu_p-x)\over f(\nu_p+x)}$ increases on $[0,c^*)$, and decreases on $(c^*,\nu_p)$. Thus, there is a point $c \in (c^*,\nu_p)$ such that ${f(\nu_p-c)\over f(\nu_p+c)}=1$ and the conditions in Lemma~\ref{lem:StochDomCrit} are satisfied.

We conclude that ${1 \over H_p}x^{p-1}\,f(\nu_p+x){\bf 1}_{(0,R-\nu_p)}(x)$ exhibits strict stochastic dominance over ${1 \over H_p}x^{p-1}\,f(\nu_p-x){\bf 1}_{(0,\nu_p -L)}(x)$.
Hence, $f(x)$ is {\it truly mode positively skewed} (Def.~\ref{def:TPSmode}) by Thm.~\ref{thm:TPSkew}.

\bigskip
\noindent
{\bf Case ii:} $0<\alpha <1$. Here, $f(x)$ is {\it truly positively skewed} (Def.~\ref{def:TPS}) by Prop.~\ref{prop:StochDomCritEZ} and Thm.~\ref{thm:TPSkew}.

\subsection{Beta distribution}\label{sec:beta}
Consider a beta random variable with p.d.f. $$f(x)={1 \over {\rm B}(\alpha,\beta)}x^{\alpha-1} (1-x)^{\beta-1}$$ 
over $(L,R)=(0,1)$, with parameters $\,\beta>\alpha >1$. 
Here, ${\rm B}(\alpha,\beta)={\Gamma(\alpha) \Gamma(\beta) \over \Gamma(\alpha+\beta)}$ denotes the {\it beta function}.
The mode equals $\,\nu_0={\alpha -1 \over \alpha+\beta -2}< {1 \over 2}$.
Differentiating 
\begin{align*}
\log f(\nu_p-x)-\log f(\nu_p+x)=&(\alpha-1)\log(\nu_p-x)-(\alpha-1)\log(\nu_p+x) \\
&~~+(\beta-1)\log(1-\nu_p+x)-(\beta-1)\log(1-\nu_p-x) 
\end{align*}
with respect to $x$ and setting it equal to zero, we obtain
$$(\alpha-1){\nu_p \over \nu_p^2-x^2}=(\beta-1){1-\nu_p \over (1-\nu_p)^2-x^2},$$
yielding
\begin{equation}\label{eqn:cMinBeta}
x^2=-(1-\nu_p)\nu_p{\alpha-1 -(\alpha+\beta-2)\nu_p \over \beta-1 -(\alpha+\beta-2)\nu_p}
=(1-\nu_p)\nu_p{\nu_p-\nu_0 \over (1-\nu_0)-\nu_p}.
\end{equation}
Recall that $\nu_0 < {1 \over 2}$.
Suppose $0<\nu_p \leq \nu_0$ for $p>0$, then \eqref{eqn:cMinBeta} has no positive real solution. 
Thus, since $f(\nu_p-\nu_p)=f(0)=0<f(2\nu_p)=f(\nu_p+\nu_p)$, we have  $f(\nu_p-x)< f(\nu_p+x)$ for all $x \in (0,1-\nu_p)$. 
Therefore, assumption $0<\nu_p \leq \nu_0$ contradicts \eqref{eqn:genNu_p_Int}. Hence, $\nu_p>\nu_0$.

\medskip
\noindent
Next, suppose $\nu_p \geq {1 \over 2}$ for $p>0$, then $1-\nu_p \leq \nu_p$ and
$$\left({1-\nu_p+x \over 1-\nu_p-x}\right)^{\beta-1} > \left({\nu_p+x \over \nu_p-x}\right)^{\alpha-1} \quad \forall x \in (0,1-\nu_p).$$
Thus, $f(\nu_p-x)> f(\nu_p+x)$ for all $x \in (0,\nu_p)$.
Hence, assumption $\nu_p \geq {1 \over 2}$ also contradicts equation \eqref{eqn:genNu_p_Int}.

\medskip
\noindent
Consequently, $\nu_0<\nu_p<{1 \over 2}$ for all $p \in \mathcal{D}$, and $c^*=\sqrt{(1-\nu_p)\nu_p{\nu_p-\nu_0 \over (1-\nu_0)-\nu_p}}$ 
is the only extremum of ${f(\nu_p-x)\over f(\nu_p+x)}$ in $(0,\nu_p)$. 
Now, ${f(\nu_p-0)\over f(\nu_p+0)}=1$ and ${f(\nu_p-\nu_p)\over f(\nu_p+\nu_p)}={f(0)\over f(2\nu_p)}=0$. 
Since by \eqref{eqn:genNu_p_Int}, ${f(\nu_p-x)\over f(\nu_p+x)}$ cannot be $\leq 1$ for all $x>0$, the extremum $c^*$ 
is the point of the maximum of ${f(\nu_p-x)\over f(\nu_p+x)}$.
Therefore, there is a point $c \in (c^*,\nu_p)$ such that ${f(\nu_p-c)\over f(\nu_p+c)}=1$ and the conditions in Lemma~\ref{lem:StochDomCrit} are satisfied.

\medskip
\noindent
We conclude that ${1 \over H_p}x^{p-1}\,f(\nu_p+x){\bf 1}_{(0,R-\nu_p)}(x)$ exhibits strict stochastic dominance over 
${1 \over H_p}x^{p-1}\,f(\nu_p-x){\bf 1}_{(0,\nu_p -L)}(x)$.
Thus, $f(x)$ is  {\it truly mode positively skewed} (Def.~\ref{def:TPSmode})  by Thm.~\ref{thm:TPSkew}.

\subsection{Log-normal distribution}\label{sec:LogNormal}
Consider a log-normal random variable with p.d.f. $$f(x)={1 \over x\sqrt{2\pi \sigma^2}}\exp\left\{-{(\log{x}-\mu)^2 \over 2\sigma^2} \right\}$$ 
over $(L,R)=(0,\infty)$, with parameters $\,\mu\,$ and $\,\sigma^2$. Here, the mode equals $\,\nu_0=\exp\left\{\mu -\sigma^2\right\}$, the median is $\,\nu_1=\exp\left\{\mu\right\}$, and the mean $\,\nu_2=\exp\left\{\mu+{\sigma^2 \over 2}\right\}$. 
Next, we find $\nu_p$ and conclude that $\nu_p$ increases as a power of $p$.
See Fig.~\ref{fig:LogNorm}.

\begin{thm}\label{thm:LogNormal}
For a log-normal random variable with parameters $\,\mu\,$ and $\,\sigma^2$,
$$\nu_p=\exp\left\{\mu+{p-1 \over 2}\sigma^2\right\}$$
solves \eqref{eqn:genNu_p_Int} for all $\,p\in (0,\infty)$.
\end{thm}
\begin{proof}
Without loss of generality, let $\mu=0$. We need to show that $\,\nu_p=\exp\left\{{p-1 \over 2}\sigma^2\right\}\,$ with 
$\,f(x)={1 \over x\sqrt{2\pi \sigma^2}}\exp\left\{-{(\log{x})^2 \over 2\sigma^2} \right\}\,$ satisfy \eqref{eqn:genNu_p_Int} for all $p\in (0,\infty)$.

\medskip
\noindent
First, letting $\,z=\log\nu_p-\log(\nu_p-x)\,$
in the left hand side of \eqref{eqn:genNu_p_Int}, we have 
\begin{align}\label{eqn:LogNormLHS}
\int\limits_0^{\nu_p} x^{p-1} &\,f(\nu_p-x)\,dx
={1 \over \sqrt{2\pi \sigma^2}}\nu_p^{p-1}\int\limits_0^\infty \big(1-e^{-z}\big)^{p-1}\exp\left\{-{\left(z-\log\nu_p \right)^2 \over 2\sigma^2} \right\}dz \nonumber \\
=&{1 \over \sqrt{2\pi \sigma^2}}\nu_p^{p-1}\int\limits_0^\infty \big(e^z-1\big)^{p-1}\exp\left\{-{z^2+(\log\nu_p)^2 \over 2\sigma^2}+\left({\log\nu_p \over \sigma^2}+1-p\right)z \right\}dz.
\end{align}
Next, we let $\,z=\log(\nu_p+x)-\log\nu_p\,$ 
in the right hand side of \eqref{eqn:genNu_p_Int}, obtaining
\begin{align}\label{eqn:LogNormRHS}
\int\limits_0^\infty  x^{p-1} &\,f(\nu_p+x)\,dx
={1 \over \sqrt{2\pi \sigma^2}}\nu_p^{p-1}\int\limits_0^\infty \big(e^z-1\big)^{p-1}\exp\left\{-{\left(z+\log\nu_p \right)^2 \over 2\sigma^2} \right\}dz \nonumber \\
=&{1 \over \sqrt{2\pi \sigma^2}}\nu_p^{p-1}\int\limits_0^\infty \big(e^z-1\big)^{p-1}\exp\left\{-{z^2+(\log\nu_p)^2 \over 2\sigma^2}-{\log\nu_p \over \sigma^2}z \right\}dz.
\end{align}
Hence, the integrals in \eqref{eqn:LogNormLHS} and \eqref{eqn:LogNormRHS} are equal if and only if $\,\nu_p=\exp\left\{{p-1 \over 2}\sigma^2\right\}$. 
\end{proof}

\begin{cor}
A log-normal random variable is truly mode positively skewed (Def.~\ref{def:TPSmode}).
\end{cor}

\medskip
\noindent
Observe that $\,\nu_2=\exp\left\{\mu+{\sigma^2 \over 2}\right\}\,$ and $\,\nu_4=\exp\left\{\mu+{3\sigma^2 \over 2}\right\}\,$ together
with the standard deviation $\,\sqrt{\exp\left\{\sigma^2\right\}-1}\exp\left\{\mu+{\sigma^2 \over 2}\right\}\,$
and the Pearson's moment coefficient of skewness $\,\gamma=\left(\exp\left\{\sigma^2\right\}+2\right)\sqrt{\exp\left\{\sigma^2\right\}-1}\,$
satisfy the equality in \eqref{eqn:cubic3}.

\subsection{Pareto distribution}\label{sec:Pareto}
For a given parameter $\,\alpha >0$, consider a random variable with p.d.f. $$f(x)={\alpha \over x^{\alpha+1} }$$ 
over $(L,R)=(1,\infty)$. Here, the mode equals $\,\nu_0=1$, and $\mathcal{D}=[1,\alpha+1)$. 
Equation \eqref{eqn:genNu_p_Int} implies
\begin{equation}\label{eqn:a1Pareto}
\int\limits_0^{\nu_p-1} {x^{p-1} \over (\nu_p-x)^{\alpha+1}}\,dx=\int\limits_0^\infty {x^{p-1} \over (\nu_p+x)^{\alpha+1}}\,dx.
\end{equation}
Substituting $z=x/\nu_p$ into \eqref{eqn:a1Pareto}, we have
\begin{equation}\label{eqn:a1ParetoZ}
Z_p:=\int\limits_0^{1-1/\nu_p} {z^{p-1} \over (1-z)^{\alpha+1}}\,dz=\int\limits_0^\infty {z^{p-1} \over (1+z)^{\alpha+1}}\,dz,
\end{equation}
where $Z_p=\alpha^{-1} \nu_p^{1+\alpha-p} H_p$. Thus, $\nu_p>1=\nu_0$.
Now, since ${z^{p-1} \over Z_p(1+z)^{\alpha+1}}{\bf 1}_{(0,\infty)}(z)$ exhibits strict stochastic dominance over 
${z^{p-1} \over Z_p(1-z)^{\alpha+1}}{\bf 1}_{(0,1-1/\nu_p)}(z)$, Lemma \ref{lem:StD} implies
\begin{equation}\label{eqn:a1ParetoLog}
\int\limits_0^\infty \log{z}\, {z^{p-1} \over (1+z)^{\alpha+1}}\,dz ~> ~ \int\limits_0^{1-1/\nu_p} \log{z}\,{z^{p-1} \over (1-z)^{\alpha+1}}\,dz.
\end{equation}
Next, we differentiate both integrals in \eqref{eqn:a1ParetoZ} with respect to $p$, obtaining
$$\nu_p^{\alpha-p}(\nu_p-1)^{p-1}{d \nu_p \over dp} +\int\limits_0^{1-1/\nu_p} \log{z}\,{z^{p-1} \over (1-z)^{\alpha+1}}\,dz=\int\limits_0^\infty \log{z}\, {z^{p-1} \over (1+z)^{\alpha+1}}\,dz.$$
Therefore, by \eqref{eqn:a1ParetoLog},
$${d \nu_p \over dp}=\nu_p^{p-\alpha}(\nu_p-1)^{1-p}\left(\int\limits_0^\infty \log{z}\, {z^{p-1} \over (1+z)^{\alpha+1}}dz \,-\, \int\limits_0^{1-1/\nu_p} \log{z}\,{z^{p-1} \over (1-z)^{\alpha+1}}dz\right) ~>0.$$ 
Hence, $f(x)$ is {\it truly mode positively skewed}.
Notice that for $\alpha \in (0,1)$, the quantities $\nu_1$, $\nu_2$, and $\gamma$ do not exist. Yet, using Def.~\ref{def:TPS} and stochastic dominance, we established the positive sign of skewness. Naturally, this was expected of a distribution with only right tail and no left tail.
This sends us back to the discussion in Remark~\ref{rem:SD}.

\section{Extending to $p \in (0,1)$}\label{sec:p01}
In this section we will show that Definition~\ref{def:meansAlt} of $p$-mean can often be extended to include all $p \in (0,1)$.
For instance, in the case when $X$ is an exponential random variable, equation \eqref{eqn:expmain} defines $\nu_p$ {\it uniquely} for all real $p>0$.

\medskip
\noindent
Similarly, in the case of a gamma random variable with parameters $\alpha, \lambda >0$, 
equation \eqref{eqn:genNu_p_Int} implies
\begin{equation}\label{eqn:a1Gamma}
\int\limits_0^{\nu_p} x^{p-1} \,(\nu_p-x)^{\alpha-1}e^{-\lambda(\nu_p-x)}\,dx=\int\limits_0^\infty x^{p-1} \,(\nu_p+x)^{\alpha-1}e^{-\lambda(\nu_p+x)}\,dx.
\end{equation}
Substituting $z=x/\nu_p$ into \eqref{eqn:a1Gamma}, we obtain
\begin{equation}\label{eqn:a1GammaZ}
Z_p:=\int\limits_0^1 z^{p-1} \,(1-z)^{\alpha-1}e^{\lambda \nu_p z}\,dz=\int\limits_0^\infty z^{p-1} \,(1+z)^{\alpha-1}e^{-\lambda \nu_p z}\,dz,
\end{equation}
where $Z_p=\nu_p^{1-\alpha-p} e^{\lambda \nu_p}H_p$.
Thus, for a gamma distribution, equation \eqref{eqn:a1GammaZ} also defines $\nu_p$ {\it uniquely} for all $p>0$.

\medskip
\noindent
For the log-normal random variables, equations \eqref{eqn:LogNormLHS} and \eqref{eqn:LogNormRHS} imply the {\it uniqueness} of $\nu_p$  for all $p>0$.
Moreover, Thm.~\ref{thm:LogNormal} finds the close form expression for the unique $\nu_p$ that solves \eqref{eqn:nupDefAlt} for all $p>0$.
Finally, for the Pareto random variables, \eqref{eqn:a1ParetoZ} also defines $\nu_p$ {\it uniquely} for all real $p>0$.

\medskip
\noindent
Even for a Bernoulli random variable with parameter $1/2$ and $\,p\in (0,1)$, $\nu_p$ defined as in Def.~\ref{def:meansAlt} is unique, while $\nu_p$ defined as in Def.~\ref{def:means} is not. 

\bigskip
\noindent
We will try to give an argument for extending Definition~\ref{def:meansAlt} to $p \in (0,1)$ in a more general way.
Suppose $f(x)$ is differentiable in $(L,R)$,  then, by \eqref{eqn:genNu_p_Int}, we have
\begin{align*}
0 =& {d \over dp}\left(\int\limits_0^{\nu_p -L} x^{p-1} \,f(\nu_p-x)\,dx ~-\int\limits_0^{R-\nu_p} x^{p-1} \,f(\nu_p+x)\,dx \right)\\
=& {d \nu_p \over dp}\left((\nu_p -L)^{p-1}f(L)+(R-\nu_p)^{p-1}f(R)+\int\limits_0^{\nu_p -L} x^{p-1} f' (\nu_p-x)dx -\int\limits_0^{R-\nu_p} x^{p-1} f' (\nu_p+x)dx \right)\\
&\quad +\int\limits_0^{\nu_p -L} x^{p-1}\, \log{x}\, f(\nu_p-x)\, dx ~-~ \int\limits_0^{R-\nu_p} x^{p-1}\,\log{x}\, f(\nu_p+x)\, dx,
\end{align*}
and therefore,
\begin{equation}\label{eqn:dnupGen}
{d \nu_p \over dp}={\int\limits_0^{R-\nu_p} x^{p-1}\, \log{x}\, f(\nu_p+x)\, dx -  \int\limits_0^{\nu_p -L} x^{p-1}\,\log{x}\, f(\nu_p-x)\, dx \over 
(\nu_p -L)^{p-1}f(L)+(R-\nu_p)^{p-1}f(R)+\int\limits_0^{\nu_p -L} x^{p-1} f' (\nu_p-x)dx -\int\limits_0^{R-\nu_p} x^{p-1} f' (\nu_p+x)dx}.
\end{equation}
Quantities $(\nu_p -L)^{p-1}f(L)$ and $(R-\nu_p)^{p-1}f(R)$ are interpreted as the corresponding left and right limits.
This includes the case when $L=-\infty$ and the case when $R=\infty$.

\medskip
\noindent
Now, by the Picard-Lindel\"{o}f  existence and uniqueness theorem \cite{Teschl2012}, if the right hand side in \eqref{eqn:dnupGen} is a Lipschitz function in $\nu_p$, 
then the solution $\nu_p$ of \eqref{eqn:dnupGen} extends uniquely to $\,p \in (0,1)$. 

\medskip
\noindent
So, if we can extend the Definition~\ref{def:meansAlt} of $p$-mean to all $p \in (0,1)$, then we can extend the definition of true positive skewness (Def.~\ref{def:TPS}).
Here, we say that a random variable (or its distribution) is {\it truly positively skewed over the full domain} if $\,\nu_p\,$ is increasing for $p$ in the domain
$\,(0,1) \cup \mathcal{D}$. In the unimodal case, we say that a distribution is {\it truly mode positively skewed over the full domain} if $\,\nu_p\,$ is increasing for $p$ in the domain $\,[0,1) \cup \mathcal{D}$.

\medskip
\noindent
In order to establish true positive skewness over the full domain, we need to show that $\,{d \nu_p \over dp}>0\,$  for all $\,p \in (0,1) \cup \mathcal{D}$. 
Often, we can show that the numerator and the denominator in  \eqref{eqn:dnupGen} are both positive.
Notice that the criteria in Lemma~\ref{lem:StochDomCrit} and Prop.~\ref{prop:StochDomCritEZ} work for $p \in (0,1)$ as well, establishing
stochastic dominance of density ${1 \over H_p}x^{p-1}\,f(\nu_p+x){\bf 1}_{(0,R-\nu_p)}(x)$ over density ${1 \over H_p}x^{p-1}\,f(\nu_p-x){\bf 1}_{(0,\nu_p -L)}(x)$.
Hence, Lem.~\ref{lem:StochDomCrit} and Prop.~\ref{prop:StochDomCritEZ} can be used to prove positivity of the numerator in the RHS of \eqref{eqn:dnupGen} via equation \eqref{eqn:main}.
Finally, we notice that the denominator in the RHS of \eqref{eqn:dnupGen} is positive if $\log{f(x)}$ is concave over $(L,R)$.
\begin{lem}\label{lem:LogConcaveTPSkew}
Consider a continuous random variable supported over $(L,R)$ with p.d.f. $f(x)$ differentiable on $supp(f)=(L,R)$.
Suppose Definition~\ref{def:meansAlt} of $p$-mean $\nu_p$ can be extended to $p \in (0,1)$, and suppose $\log{f(x)}$ is concave.
If a random variable with density function ${1 \over H_p}x^{p-1}\,f(\nu_p+x){\bf 1}_{(0,R-\nu_p)}(x)$ exhibits strict stochastic dominance over 
a random variable with density function ${1 \over H_p}x^{p-1}\,f(\nu_p-x){\bf 1}_{(0,\nu_p -L)}(x)$,
then function $\nu_p$ is increasing at $p \in (0,1)$.
\end{lem}
\begin{proof}
Suppose $\log{f(x)}$ is concave, then ${f'(x) \over f(x)}={d \over dx}\log{f(x)}$ is nonincreasing, and 
\begin{align}\label{eqn:concaveNN}
\int\limits_0^{\nu_p -L} &x^{p-1} f' (\nu_p-x)dx -\int\limits_0^{R-\nu_p} x^{p-1} f' (\nu_p+x)dx \nonumber \\
&= \int\limits_0^{\nu_p -L} x^{p-1} {f' (\nu_p-x) \over f(\nu_p-x)} f(\nu_p-x) dx -\int\limits_0^{R-\nu_p} x^{p-1} {f' (\nu_p+x) \over f(\nu_p+x)} f(\nu_p+x)dx \nonumber \\
&\geq  ~~{f' (\nu_p) \over f(\nu_p)}\left(\int\limits_0^{\nu_p -L} x^{p-1} f(\nu_p-x) dx -\int\limits_0^{R-\nu_p} x^{p-1} f(\nu_p+x)dx\right) ~=0
\end{align}
by \eqref{eqn:genNu_p_Int}.
If $\log{f(x)}$ is not constant in $(L,R)$, then the inequality in \eqref{eqn:concaveNN} is strict.
If $\log{f(x)}$ is constant in $(L,R)$, then $L$ needs to be finite, and therefore
$$(\nu_p -L)^{p-1}f(L)>0.$$
In either case, 
\begin{equation}\label{eqn:concavePos} 
(\nu_p -L)^{p-1}f(L)+(R-\nu_p)^{p-1}f(R)+\!\!\!\!\int\limits_0^{\nu_p -L} \!\!x^{p-1} f' (\nu_p-x)dx -\!\!\!\!\int\limits_0^{R-\nu_p} \!\!\!x^{p-1} f' (\nu_p+x)dx >0.
\end{equation}
\end{proof}

\medskip
\noindent
Additionally, for $p=1$,  integration yields
$$f(L)+f(R)+\int\limits_0^{\nu_1 -L}  f' (\nu_1-x)dx -\int\limits_0^{R-\nu_1} f' (\nu_1+x)dy=2f(\nu_1)>0,$$
whence, from equations \eqref{eqn:main} and \eqref{eqn:dnupGen}, we have
\begin{equation}\label{eqn:dnupSimpleP1}
{d \nu_p \over dp}\Bigg|_{p=1}={1 \over 2f(\nu_1)}\left(\int\limits_0^{R-\nu_1} \log{x}\, f(\nu_1+x)\, dx -  \int\limits_0^{\nu_1 -L} \log{x}\, f(\nu_1-x)\, dx\right)~>0.
\end{equation}

\medskip
\noindent
Next, we go through some examples of true positive skewness over the full domain.
\begin{ex}[Exponential distribution]\label{ex:exp}
For an exponential random variable, equations \eqref{eqn:expmain1} and \eqref{eqn:expmain2} imply {\it true mode positive skewness over the full domain}.
\end{ex}

\begin{ex}[Gamma distribution]\label{ex:gamma}
Consider a gamma distribution. If $\alpha>1$, then, the calculations in Sect.~\ref{sec:beta} yield 
Lemma~\ref{lem:StochDomCrit} is satisfied for all $\,p \in (0,1) \cup \mathcal{D}$.
Therefore, density function ${1 \over H_p}x^{p-1}\,f(\nu_p+x){\bf 1}_{(0,R-\nu_p)}(x)$ exhibits strict stochastic dominance over 
density function ${1 \over H_p}x^{p-1}\,f(\nu_p-x){\bf 1}_{(0,\nu_p -L)}(x)$.
Moreover,
$${d \over dx}\log f(x) ={f'(x) \over f(x)}={\alpha -1 \over x}-\lambda$$
is a decreasing function over $(0,\infty)$ as $\alpha >1$. Hence, $\log f(x)$ is concave, and Lem.~\ref{lem:LogConcaveTPSkew} implies 
{\it true mode positive skewness over the full domain}.

\medskip
\noindent
If $0<\alpha <1$, equation \eqref{eqn:a1GammaZ} implies that
${1 \over Z_p}z^{p-1} \,(1+z)^{\alpha-1}e^{-\lambda \nu_p z}{\bf 1}_{(0,\infty)}(z)$ exhibits strict stochastic dominance over 
${1 \over Z_p}z^{p-1} \,(1-z)^{\alpha-1}e^{\lambda \nu_p z}{\bf 1}_{(0,\nu_p)}(z)$, and therefore,  Lemma \ref{lem:StD} implies
\begin{equation}\label{eqn:a1GammaLog}
\int\limits_0^\infty (\log{z})\, z^{p-1} \,(1+z)^{\alpha-1}e^{-\lambda \nu_p z}\,dz ~> ~ \int\limits_0^1 (\log{z})\,z^{p-1} \,(1-z)^{\alpha-1}e^{\lambda \nu_p z}\,dz.
\end{equation}
We differentiate both integrals in \eqref{eqn:a1GammaZ} with respect to $p$, obtaining
\begin{align*}
\int\limits_0^1 (\log{z})\,z^{p-1} &\,(1-z)^{\alpha-1}e^{\lambda \nu_p z}\,dz
+{d \nu_p \over dp}\int\limits_0^1 z^{p} \,(1-z)^{\alpha-1}e^{\lambda \nu_p z}\,dz \\
=&\int\limits_0^\infty (\log{z})\, z^{p-1} \,(1+z)^{\alpha-1}e^{-\lambda \nu_p z}\,dz
-{d \nu_p \over dp}\int\limits_0^\infty z^{p} \,(1+z)^{\alpha-1}e^{-\lambda \nu_p z}\,dz.
\end{align*}
Thus, for all $p>0$, we have
$${d \nu_p \over dp}={\int\limits_0^\infty (\log{z})\, z^{p-1} \,(1+z)^{\alpha-1}e^{-\lambda \nu_p z}\,dz -\int\limits_0^1 (\log{z})\,z^{p-1} \,(1-z)^{\alpha-1}e^{\lambda \nu_p z}\,dz
\over \int\limits_{-1}^\infty |z|^{p} \,(1+z)^{\alpha-1}e^{-\lambda \nu_p z}\,dz} ~>0$$
by \eqref{eqn:a1GammaLog}.
Hence, $f(x)$ is {\it truly positively skewed over the full domain}.
\end{ex}

\begin{ex}[Beta distribution]\label{ex:beta}
For a beta distribution with parameters $\beta>\alpha>1$, Sect.~\ref{sec:beta} yields Lemma~\ref{lem:StochDomCrit} is satisfied.
Hence, ${1 \over H_p}x^{p-1}\,f(\nu_p+x){\bf 1}_{(0,R-\nu_p)}(x)$ exhibits strict stochastic dominance over 
${1 \over H_p}x^{p-1}\,f(\nu_p-x){\bf 1}_{(0,\nu_p -L)}(x)$.
Next,
$${d \over dx}\log f(x) ={f'(x) \over f(x)}={\alpha -1 \over x}-{\beta -1 \over 1-x}$$
is a decreasing function over $(0,1)$ as both, $\,\alpha>1\,$ and $\,\beta >1$. Hence, $\log f(x)$ is concave,  and Lem.~\ref{lem:LogConcaveTPSkew} implies 
{\it true mode positive skewness over the full domain}.

\end{ex}

\begin{ex}[Log-normal distribution]\label{ex:LogNormal}
Consider a log-normal distribution. Observe that Thm.~\ref{thm:LogNormal} works for all real $p>0$.
Hence, a log-normal distribution is {\it truly mode positively skewed over the full domain}.
\end{ex}

\begin{ex}[Pareto distribution]\label{ex:Pareto}
For a Pareto distribution, calculations in Sect.~\ref{sec:Pareto} stay valid for all $p\in (0, \alpha+1)$. 
Thus, it is {\it truly mode positively skewed over the full domain}.
\end{ex}

\medskip
\noindent
Notice that even in the case of a unimodal continuous random variable with probability density function (p.d.f.) in $C^\infty$, the mode $\nu_0$ does not have to be equal $\lim\limits_{p \downarrow 0} \nu_p$.
Indeed, the case of log-normal random variables is an example of discontinuity of $\nu_p$ at $0$ as
\begin{equation}\label{eqn:LogNormDisc}
\nu_0=\exp\left\{\mu -\sigma^2\right\} ~<~\exp\left\{\mu -{\sigma^2 \over 2}\right\}=\lim\limits_{p \downarrow 0} \nu_p.
\end{equation}

\section{Discussion}\label{sec:dis}

While the unification of the three classical characterizations of skewness introduced here is important in its own right, the main insight of 
true positive/negative skewness was articulated in Remark \ref{rem:SD} following Theorem~\ref{thm:TPSkew}.
Specifically, for $p \in \mathcal{D}$, let us break the probability density function $f(x)$ into the left and right parts, $f(\nu_p-x){\bf 1}_{(0,\nu_p -L)}(x)$ and $f(\nu_p+x){\bf 1}_{(0,R-\nu_p)}(x)$. We multiply the two parts by $\,{1 \over H_p}x^{p-1}\,$ to make the two probability density functions, ${1 \over H_p}x^{p-1}\,f(\nu_p-x){\bf 1}_{(0,\nu_p -L)}(x)$ and ${1 \over H_p}x^{p-1}\,f(\nu_p+x){\bf 1}_{(0,R-\nu_p)}(x)$. It is natural to expect that if $f(x)$ is positively skewed, the right tail p.d.f. would exhibit stochastic dominance over the left tail p.d.f. for all $p \in \mathcal{D}$.
Theorem~\ref{thm:TPSkew} asserts that this stochastic dominance holds if the distribution is truly positively skewed.
The proposed principle of the dominating left tail p.d.f. over the right tail p.d.f. works regardless of the size of the domain $\mathcal{D}$, and
the main advantage of this new approach is that true positive (or negative) skewness can be established even in the case of infinite first moment when all three 
of Pearson's skewness measures (and even their numerators) are undefined.
This was demonstrated in Sect.~\ref{sec:Pareto} for Pareto distribution with $\alpha \in (0,1)$, while
the most compelling example was established in the work \cite{REU21} of students advised by the author of this current manuscript, where it was shown that 
L\'{e}vy distribution is truly positively skewed. This result will be published in a separate paper.

\bigskip
\noindent
Now, we will proceed by mentioning some precursors of the approach to skewness presented in this current paper.
There is a relation between the results in \cite{Zwet79} and this current paper. In \cite{Zwet79}, van Zwet proves that  
the mean-median-mode inequality
$$\mathrm{mode}~\leq~\mathrm{median}~\leq~\mathrm{mean}$$
holds whenever the cumulative distribution function $F(x)$ satisfies
\begin{equation}\label{eqn:vanZweit}
F(\nu_1-x)+F(\nu_1+x) \geq 1  \qquad \forall x>0.
\end{equation}
Observe that for a continuous unimodal random variable with p.d.f. $f(x)$, condition \eqref{eqn:vanZweit} is equivalent to
$${1 \over H_1} \int\limits_0^x f(\nu_1+y)\,dy ~\geq ~{1 \over H_1} \int\limits_0^x f(\nu_1-y)\,dy,$$
where $\,H_1={1 \over 2}$ for the median $(p=1)$.
In other words, the mean-median-mode inequality holds whenever p.d.f. ${1 \over H_1}f(\nu_1+x){\bf 1}_{(0,R-\nu_1)}(x)$ exhibits stochastic dominance over 
p.d.f. ${1 \over H_1}f(\nu_1-x){\bf 1}_{(0,\nu_1 -L)}(x)$.
That is, the stochastic dominance condition required for the true positive skewness criteria established in Sect.~\ref{sec:StocDom} holds for $p=1$.
This connection of \eqref{eqn:vanZweit} to stochastic dominance was also noticed by Dharmadhikari and Joag-dev \cite{DJd83,DJd88}.
Observe that in the unimodal case, the above result of van Zwet \cite{Zwet79} (with strict inequalities) together with Thm.~\ref{thm:TPSkew}
imply that if ${1 \over H_p}x^{p-1}\,f(\nu_p+x){\bf 1}_{(0,R-\nu_p)}(x)$ exhibits strict stochastic dominance over 
${1 \over H_p}x^{p-1}\,f(\nu_p-x){\bf 1}_{(0,\nu_p -L)}(x)$ for all $\,p \in \mathcal{D}$, then the distribution is
{\it truly mode positively skewed}.

\bigskip
\noindent
As we have seen, the notion of true skewness (i.e., monotonicity of $\nu_p$) is based on the comparison of the left and right halves of 
a continuous distribution dissected at the centroids $\nu_p$ for $p \in \mathcal{D}$. As such, true mode positive skewness is a more restrictive property 
than positivity of all three Pearson's skewness metrics. For instance, in \cite{Abadir2005} there is an example of a unimodal continuous random variable
with positive Pearson's moment coefficient of skewness $\gamma$ and negative median skewness. 
Naturally, in this example $\nu_2<\nu_1$  while Prop.~\ref{prop:first} yields $\nu_2<\nu_4$. Thus, the distribution is not truly positively or negatively skewed.

\bigskip
\noindent
Finally, Oja \cite{Oja1981} used convexity of order $k$ notion to extend the convex transformation approach to skewness and kurtosis developed in van Zwet \cite{Zwet64}.

\bigskip
\noindent
Next, we outline some potentially advantageous directions spinning out of this work.
First, in this paper, Definition \ref{def:TPS} only speaks of the sign of skewness (i.e., positive or negative).
Yet, we would like to consider measures of skewness based on the trajectory of $\nu_p$. 
One may consider the following approach generating a family of skewness measures.
For a probability density function $\varphi$ over $\,[0,\infty)$, let
$$\ell_\varphi={1 \over {d \nu_p \over dp}\big|_{p=1}}\,{\int\limits_{\mathcal{D}} \varphi(p-1)\, d \nu_p  \over \int\limits_{\mathcal{D}} \varphi(p-1)\, dp}.$$
Then, for truly positively/negatively skewed continuous distributions, positive quantities
$\ell_\varphi$, if finite, may be used to measure the magnitude of skewness.

In the case of log-normal distribution, Pearson's moment coefficient of skewness equals $\,\gamma=\left(\exp\left\{\sigma^2\right\}+2\right)\sqrt{\exp\left\{\sigma^2\right\}-1}$.
Thus, $\,\gamma$ is an increasing function of $\sigma^2$, reflecting evident increase in skewness as one increases the value of $\sigma^2$.
At the same time, contrary to the observed evolution of log-normal density (see Fig.~\ref{fig:LogNormSkew}), 
Pearson’s first and second skewness coefficients would decrease down to zero as $\sigma^2 \to \infty$. 
See \cite{AG95} for a relevant discussion.
Now, Thm.~\ref{thm:LogNormal} yields  
$$\ell_\varphi={1 \over {d \nu_p \over dp}\big|_{p=1}}\,{\int\limits_{\mathcal{D}} \varphi(p-1)\, {d \nu_p \over dp}\,dp  \over \int\limits_{\mathcal{D}} \varphi(p-1)\, dp}
=\int\limits_0^\infty \varphi(x) e^{x\sigma^2/2}dx.$$
Observe that in this case, if $\,\lim\limits_{x \to \infty}{-\ln\varphi(x) \over x}=\infty$, then $\ell_\varphi$ is a well-defined quantity that increases with increasing $\sigma^2$.

\bigskip
\noindent
Second, in a multidimensional case, for a random vector $X \in \mathbb{R}^d$, consider Fr\'{e}chet $p$-means $\nu_p \in \mathbb{R}^d$ 
as defined in Def.~\ref{eqn:p-mean}.
We believe that the trajectory of $\nu_p$ can be interpreted as a tailbone of the distribution, or the {\it trajectory of skewness}.
Moreover, if the asymptotic limit 
$$\zeta=\lim\limits_{p \to \infty} \left({d \nu_p \over dp}\Bigg/ \left\| {d \nu_p \over dp}\right\|\right),$$
exists, then $\zeta$ can be interpreted as the asymptotic direction of skewness in $\mathbb{R}^d$.
Estimating this tailbone trajectory and its limit $\zeta$ can be done numerically.

\begin{figure}[t] 
\centering\includegraphics[width=0.98\textwidth]{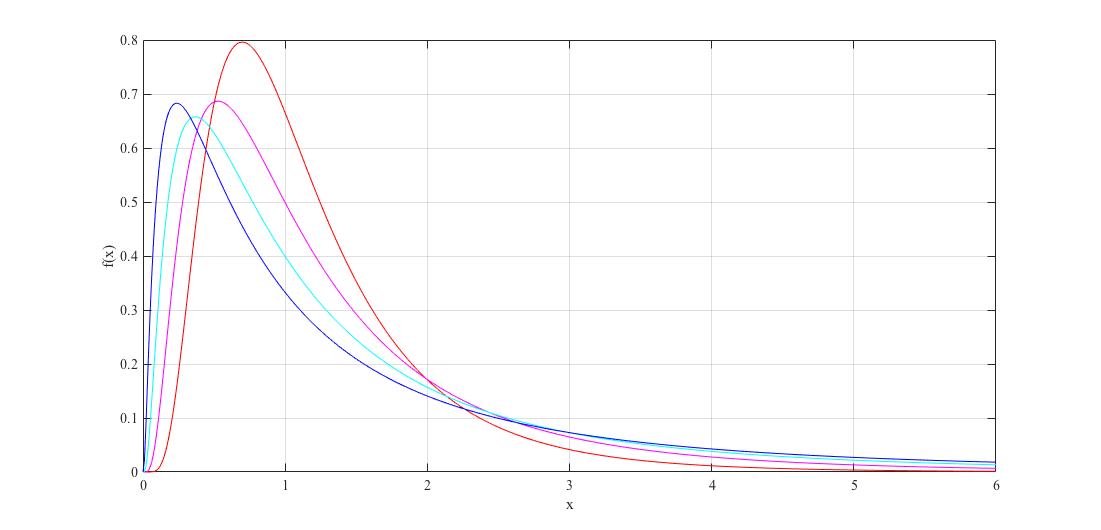}
\caption[Skewness of log-normal density function]{The log-normal density functions with different parameter values: $(\mu=0,\,\sigma=0.6)$ in red, $(\mu=0,\,\sigma=0.8)$ in magenta, $(\mu=0,\,\sigma=1)$ in light blue, $(\mu=0,\,\sigma=1.2)$ in dark blue. Evidently, as parameter $\mu$ does not affect the skewness, the distribution skews more to the right with greater values of $\sigma$.}
\label{fig:LogNormSkew}
\end{figure}

\section*{Acknowledgements}
The author would like to thank Anatoly Yambartsev, Ilya Zaliapin, Jordan Stoyanov, and the anonymous referee for their helpful comments and suggestions.
This research was supported by FAPESP award 2018/07826-5 and by NSF award DMS-1412557.


\end{document}